\documentclass[a4paper,12pt]{article}

\usepackage{amsthm,amsmath,stmaryrd,bbm,hyperref,geometry,color}
\usepackage[utf8]{inputenc}
\usepackage{amssymb}
\usepackage[french]{babel}
\usepackage{graphicx}
\usepackage{amsfonts,amssymb}
\usepackage{verbatim}
\usepackage{enumitem}

\newcommand{\po}{\left(}
\newcommand{\pf}{\right)}

\newcommand{\cco}{\llbracket}
\newcommand{\ccf}{\rrbracket}
\newcommand{\pso}{\left\langle}
\newcommand{\psf}{\right\rangle}
\newcommand{\R}{\mathbb R} 
\newcommand{\T}{\mathbb T} 
\newcommand{\Z}{\mathbb Z} 
 
\newcommand{\N}{\mathbb N} 
\newcommand{\dd}{\text{d}}
\newcommand{\na}{\nabla}
\newcommand{\1}{\mathbbm{1}}

\newtheorem{thm}{Théorème}
\newtheorem{assu}{Hypoth\`ese}

\newtheorem{cor}[thm]{Corollary}

\title{Hypocoercivité $L^2$, inégalité de concentration,  temps d'atteinte et fonctions de Lyapunov.}
\author{Pierre Monmarché}

\begin{document}
\maketitle

\abstract{On montre que, pour un semi-groupe de Markov, l'hypocoercivité $L^2$ -- c'est-à-dire la contractivité d'une norme $L^2$ modifiée -- implique des inégalités de concentration quantitatives et l'intégrabilité exponentielle des temps d'atteinte des ensembles de mesure positive. D'autre part, pour les diffusions et sous une hypothèse forte d'hypoellipticité, on établit que l'hypocoercivité $L^2$ implique l'existence d'une fonction de Lyapunov pour le générateur associé. Une traduction en anglais est disponible.
%We establish that, for a Markov semi-group, $L^2$ hypocoercivity, i.e. contractivity for a modified $L^2$ norm, implies quantitative deviation bounds for additive functionals of the associated Markov process and exponential integrability of the hitting time of sets with positive measure. Moreover, in the case of diffusion processes and under a strong hypoellipticity assumption, we prove that $L^2$ hypocoercivity implies the existence of a Lyapunov function for the generator.
}

\section{Introduction}

%Motivation : lien entre hypocoercivité (en particulier à la Dolbeault-Mouhot-Schmeiser \cite{DMS2011}) et Meyn-Tweedie/Foster-Lyapunov \cite{HairerMattingly2008}. C'est le genre de question considérée dans le cas réversible dans \cite{Cattiaux2008,CattiauxGuillinPAZ}. On se rend compte en voulant l'adapter que, via \cite{CattiauxGuillin}, c'est basé sur les déviations de \cite{Wu}. 

La motivation première de cette note est la comparaison entre deux méthodes d'obtention de taux de convergence à l'équilibre pour les processus de Markov: d'un côté, les inégalités fonctionnelles et les méthodes d'entropie et, de l'autre, l'approche probabiliste classique à la Meyn-Tweedie, basée sur un critère de Foster-Lyapunov et une condition de Doeblin locale ou de couplage. Le lien entre ces deux types d'arguments est étudié par Cattiaux, Guillin et leurs co-auteurs dans une série d'articles \cite{Cattiaux2008,CattiauxGuillin3,CattiauxGuillin2,CattiauxGuillinPAZ}, dans un cadre principalement restreint aux processus de Markov réversibles, avec une mesure d'équilibre qui satisfait typiquement une inégalité de Poincaré (par rapport à la forme de Dirichlet associée au processus). Néanmoins, dans la dernière décennie, les méthodes hypocoercives d'entropies modifiées se sont avérées très efficaces pour traiter des dynamiques non réversibles, non elliptiques ou non diffusives avec des arguments d'inégalités fonctionnelles. Nous nous concentrerons ici sur la méthode de Dolbeault-Mouhot-Schmeiser (DMS) \cite{DMS2011}, qui fournit des résultats d'hypocoercivité en norme $L^2$. Un des attraits de celle-ci est qu'elle décrit une construction générale et systématique de norme modifiée, indépendemment du processus étudié. La situation est assez différente pour l'approche Meyn-Tweedie qui nécessite la construction d'une fonction de Lyapunov adaptée à la dynamique, ce qui se révlèe parfois délicat pour des processus en un sens dégénérés. Une bonne illustration en est donnée par la comparaison des deux travaux \cite{Andrieu} et \cite{DurmusGuillinMonmarche}, qui s'intéressent tout deux à un processus cinétique de saut (le \emph{Bouncy Particle Sampler}), déterministe par morceau, respectivement par l'approche DMS et Meyn-Tweedie. Avec cette dernière, la construction de la fonction de Lyapunov est relativement pénible et en particulier induit des hypothèses techniques restrictives sur la log-densité de la mesure d'équilibre. En revanche, dans la méthode DMS, la construction est standardisée et mène à un résultat sous des conditions plus claires et générales sur la log-densité. Une question naturelle se pose donc: dans les cas d'une convergence (hypocoercive) exponentiellement rapide à l'équilibre dans $L^2$, existe-t-il toujours une fonction de Lyapunov au sens de Meyn-Tweedie ? La réponse est positive dans le cas des diffusions réversibles elliptiques \cite{CattiauxGuillinPAZ}, pour lesquelles la décroissance dans $L^2$ a lieu avec la norme usuelle. Ce résultat est basé sur l'intégrabilité exponentielle des temps d'atteinte du processus, elle-même obtenue à partir d'inégalités de concentration sur les fonctionnelles additives du processus démontrée dans \cite{Wu,CattiauxGuillin}. Ces deux résultats sont en fait intéressant en eux-même; en fait, dans certains cas, obtenir des estimations sur les temps d'atteinte est la question principale et une fonction de Lyapunov n'est qu'un outil intermédiaire pour les obtenir. Notre résultat principal est que, dès lors que la méthode DMS s'applique, alors ces deux résultats ont cours. De plus, à partir de l'intégrabilité exponentielle des temps d'atteinte, dans le cas d'une diffusion fortement hypoelliptique, on construit une fonction de Lyapunov pour le générateur associé, ce qui répond à la question initiale. En fait, comme nous l'avons appris après la rédaction de cette note, les inégalités de concentration ont été établies dans le preprint récent \cite{ReyBellet}. À notre connaissance, les autres résultats sont nouveaux.

\section{Resultats et démonstrations}

Soit $(X_t)_{t\geqslant 0}$  un processus de Markov conservatif à temps continu sur un espace Polonais $E$, avec une mesure de probabilité invariante $\mu$. On note $(P_t)_{t\geqslant 0}$ et $L$ le semi-groupe sur $L^2(\mu)$  associé et son générateur. On suppose que $L$ est fermé de domaine $D(L)$  dense dans $L^2(\mu)$. On note $\|\cdot\|_2$ et $\langle\cdot\rangle$ la norme et le produit scalaire usuels dans $L^2(\mu)$. Notre hypothèse principale est la suivante: 

\begin{assu}\label{HypoDMS}
Il existe $\rho>0$ et un opérateur linéaire symmétrique borné $S$ sur $L^2(\mu)$ tel que pour tout $f\in D(L)$,
\begin{eqnarray*}
\| S f\|_2 & \leqslant & \frac12 \| f-\mu f\|_2 \\
\langle f,Lf + SLf\rangle  & \leqslant & - \rho \| f-\mu f\|_2^2\,.
\end{eqnarray*}
\end{assu}
L'existence d'un tel $S$, que nous utiliserons dans cette note comme une boîte noire, est en fait l'outil principal dans la méthode DMS. Plus précisément, en consiérant $\varepsilon\in(0,1)$ et l'opérateur $A$ utilisé pour définir l'entropie modifiée $H$ dans  \cite{DMS2011}, alors nous posons $S=\varepsilon(\rm{Id}-\mu)( A +  A^*)(\rm{Id}-\mu)/2$. Il est alors direct de vérifier que, sous les hypothèses $\mathrm{(H_1)}-\mathrm{(H_4)}$ de \cite{DMS2011}, cet opérateur $S$ convient pour l'Hypothèse~\ref{HypoDMS}. Une fois un tel opérateur $S$ obtenu, la méthode DMS pour obtenir une décroissance hypocoercivie dans $L^2(\mu)$ est la suivante: notons $B=$Id$+S$ et considérons le produit scalaire et la norme hilbertienne
\[\langle f,g\rangle_B \ = \ \langle f, Bg\rangle\,,\qquad \|f\|_B = \sqrt{\langle f,f\rangle_B}\,. \]
Cette dernière est équivalente à la norme usuelle de $L^2(\mu)$, plus précisément $1/2\|f\|_2^2 \leqslant \| f\|_B^2 \leqslant 3/2\| f\|_2^2$.  L'Hypothèse~\ref{HypoDMS} implique pour pour tout $f\in D(L)$ telle que $\mu f = 0$,
\[\pso f,Lf\psf_B \ \leqslant \ -\rho \|f\|_2^2 \ \leqslant \ - \frac{2\rho}{3} \|f\|_B^2\,.\]
En d'autres termes, $(e^{2\rho t/3}P_t)_{t\geqslant 0}$ est dissipatif sur $\{f\in L^2(\mu),\ \mu f=0\}$ muni du produit scalaire $\pso\cdot\psf_B$ de sorte que, par le théorème de    Lumer-Philips  \cite[Chapitre IX, p.250]{Yosida}, pour tout $f\in L^2(\mu)$ telle que $\mu f= 0$,
\[\|P_t f\|_B \ \leqslant\ e^{-2\rho t/3}\| f\|_B\,. \]
En conclusion, on obtient la décroissance hypocoercive de la norme usuelle de $L^2(\mu)$:
\[\| P_t f - \mu f\|_2 \ \leqslant \ \sqrt 2 \|P_t f -\mu f\|_B \ \leqslant \ \sqrt 2e^{-2\rho t/3}\|f-\mu f\|_B \ \leqslant \  \sqrt 3 e^{-2\rho t/3}\|f-\mu f\|_2\,.\]

En fait, étant donnée une fonction mesurable bornée $V$ sur $E$, le même argument s'applique au semi-groupe de Feynman-Kac $(P_t^V)_{t\geqslant 0}$ sur $L^2(\mu)$ défini par
\[P_t^V f(x) \ = \ \mathbb E_x \po f(X_t) e^{\int_0^t V(X_s)\dd s}\pf\,,\]
où l'indice $x$ précise la condition initiale du processus.  En effet, en notant
\[\Lambda(V) \ = \ \sup\left\{ \pso f, (L+V)f \psf_B \,:\, f\in D(L)\ ,\ \|f\|_B = 1 \right\}\,,\]
on obtient que $(e^{-t\Lambda(V)}P_t^V)_{t\geqslant 0}$ est dissipatif sur $(L^2(\mu),\pso\cdot\psf_B)$ et donc que, pour tout $f\in L^2(\mu)$,
\begin{eqnarray}\label{EqPtV}
\| P_t^V f\|_B & \leqslant  & e^{t\Lambda(V)} \|f\|_B\,.
\end{eqnarray}
De ceci, en suivant \cite{Wu} et \cite{CattiauxGuillin}, on obtient les inégalités de concentration suivantes

\begin{thm}\label{TheoremDeviation}
Sous l'Hypothèse \ref{HypoDMS}, soit $\nu \ll \mu$ une mesure de probabilité sur $E$ et $V$ une fonction mesurable bornée sur $E$. Alors, pour tout $t\geqslant 0$ et tout $r\geqslant 0$, 
\begin{eqnarray*}
\mathbb P_\nu \po \frac1t \int_0^t V(X_s) \dd s -\mu V \geqslant r \pf & \leqslant &  \sqrt 2 \left\| \frac{\dd \nu}{\dd \mu}\right\|_2 e^{-t h(r)} \,,
\end{eqnarray*}
où
\[h(r) = \frac{\rho r^2}{25\|V\|_2^2 + 6 \| V\|_\infty r} \,. \]
\end{thm}
%
%More explicitly, denoting $\lambda_0 = \rho/(3\|V\|_\infty)$ and $r_0 = 25\|V\|_2^2/(2\|V\|_\infty)$, 
%\[h(r) \ = \ \frac{\rho (r\wedge r_0)^2}{25\|V-\mu V\|_2^2} + \lambda_0 \max(r-r_0,0)\,.\]
%\left\{\begin{array}{ll}\frac{\rho r^2}{25\|V-\mu V\|_2^2} & \text{if }r\leqslant r_0 \\& \\\lambda_0 r - \frac{25}{4\rho}\lambda_0^2\|V-\mu V\|_2^2  & \text{otherwise.}\end{array}\right. 
%and in all cases
%\[h(r) \ \geqslant \ \sup_{\lambda \in[0,\lambda_0]} \left\{ \lambda (r\wedge r_0) - \frac{25\lambda^2\|V-\mu V\|_2^2}{4\rho}\right\} \ = \ \frac{\rho (r\wedge r_0)^2}{25\|V-\mu V\|_2^2}\]

Évidemment en changeant $V$ en $-V$, ceci fournit des intervalles de confiance non-asymptotiques pour la moyenne empirique des fonctions bornées.

Dans la mesure où les estimations obtenues par DMS ne sont en général pas optimale, nous avons préféré ici donner une expression simple pour $h(r)$ plutôt que la plus grande possible.

Bien qu'un résultat similaire soit déjà établi dans \cite{ReyBellet}, nous en laissons la preuve, dans la mesure où elle est élémentaire, brève et intéressante.

\begin{proof}
Sans perte de généralité on suppose que $\mu V = 0$. Par les inégalités de Chebyshev  et de Cauchy-Schwarz et en utilisant \eqref{EqPtV}, pour tout $\nu \ll \mu$ et $\lambda\geqslant 0$,
\begin{eqnarray}
\mathbb P_\nu \po \frac1t \int_0^t V(X_s) \dd s  \geqslant r \pf & \leqslant & e^{-\lambda t r} \mathbb E_\nu \po e^{\lambda\int_0^t V(X_s)\dd s} \pf \notag\\
& =& e^{-\lambda t r} \int_E P_t^{\lambda V} \1 \dd \nu\notag \\
& \leqslant & \sqrt 2 e^{-\lambda t r}  \left\| \frac{\dd \nu}{\dd \mu}\right\|_2 \left\|P_t^{\lambda V}\1\right \|_B\notag\\
& \leqslant & \sqrt 2 \left\| \frac{\dd \nu}{\dd \mu}\right\|_2 e^{-\lambda t r + t \Lambda(\lambda V)} \label{EqDeviation}\,.
\end{eqnarray}
On borne alors $\Lambda(\lambda V)$ dans l'esprit de \cite{CattiauxGuillin}. D'abord, par l'Hypothèse~\ref{HypoDMS},  
\begin{eqnarray*}
\Lambda(V) & \leqslant & 2 \sup\left\{ -\rho \| f-\mu f\|^2  +  \pso f, V f \psf_B \,:\, f\in D(L)\ ,\ \|f\|_2 = 1 \right\}
\end{eqnarray*} 
Fixons $f\in D(L)$ avec $\|f\|_2 = 1$ et soit $\gamma \geqslant 0$ donné par $1+\gamma^2 = 1/(\mu f)^2$, de sorte que
\[g \ := \ \frac1\gamma \po \sqrt{1+\gamma^2} f - 1\pf \ = \ \frac{\sqrt{1+\gamma^2}}{\gamma}(f-\mu f)\]
satisfait $\mu g = 0$ et
\[\mu( g^2) \ = \ \frac{1+\gamma^2}{\gamma^2}  \po \mu(f^2) - (\mu f)^2 \pf \ = \ 1\,.\]
%\begin{eqnarray*}
%\| g\|_B^2 & = & \pso g,g\psf + \pso g,Ag\psf\\
%& = & \frac{1+\gamma^2}{\gamma^2} \po \pso f,f\psf - \po \mu f\pf^2 +   \pso f,A f\psf\pf\\
%& = & \frac{1+\gamma^2}{\gamma^2} \po  \| f\|_B^2 - \frac1{1+\gamma^2}\pf \ = \ 1\,.
%\end{eqnarray*}
D'autre part,
\begin{eqnarray*}
\pso f, V f \psf_B & = & \int V f^2 \dd \mu + \pso A (f-\mu f),(Vf-\mu(Vf))\psf\\
 &  = & \int V f^2 \dd \mu + \pso A (f-\mu f),V(f-\mu f)\psf +   \pso A (f-\mu f),(V-\mu V)\psf \mu f \\
 & = & \frac{1}{1+\gamma^2} \int V \po \gamma^2 g^2 + 2 \gamma g +1 \pf \dd \mu + \frac{\gamma^2}{1+\gamma^2} \pso A g,V g\psf + \frac{\gamma}{1+\gamma^2}\pso A g,V \psf\\
 & \leqslant& \frac{3\gamma^2}{2(1+\gamma^2)}\| V\|_\infty + \frac{5\gamma}{2(1+\gamma^2)}\| V\|_2 \,,
\end{eqnarray*}
où nous avons utilisé l'inégalité de Cauchy-Schwarz et le fait que $\|Ag\|\leqslant 1/2$. En conséquence,
\[\pso f, (L+V) f \psf_B \ \leqslant\  \frac{\gamma}{1+\gamma^2} \po -\gamma\rho + \frac{3\gamma}{2}\| V\|_\infty + \frac52\| V\|_2\pf\,,\]
et pour tout $\lambda\geqslant 0$,
\begin{eqnarray*}
 \Lambda(V) & \leqslant & 2\sup_{\gamma\geqslant 0} \left\{ \frac{\gamma}{1+\gamma^2} \po -\gamma\rho + \frac{3\gamma}{2}\| V\|_\infty + \frac52\| V\|_2\pf \right\}\\
&  \leqslant & 2 \sup_{\gamma\geqslant 0} \left\{  \gamma \po -\gamma\rho + \frac{3\gamma}{2}\| V\|_\infty + \frac52\| V\|_2\pf \right\}\\
& = & \frac{25\|V\|_2^2}{4\rho-6\|V\|_\infty} \qquad \text{si }3\|V\|_\infty < 2\rho\,,\qquad +\infty\qquad\text{sinon.}
\end{eqnarray*}
En particulier, en notant
\[\lambda_0 := 2\rho/(3\|V\|_\infty)\,,\qquad \beta = \frac{25 \|V\|_2^2}{6 \|V\|_\infty } \,,\]
on obtient
%\begin{eqnarray*}
%\frac12 \Lambda(\lambda V) & \leqslant & \sup_{\gamma\geqslant 0} \left\{ \frac{\gamma}{1+\gamma^2} \po -\frac12\gamma\rho  + \frac52 \lambda\| V\|_2\pf \right\} \ \leqslant \ \frac{25\lambda^2\|V\|_2^2}{4\rho}\,,
%%& \leqslant & \sup_{\gamma\geqslant 0} \left\{ \gamma \po -\gamma\rho + \frac{3\gamma}{2}\| V\|_\infty + \frac52\| V\|_2\pf \right\} \\
%%&= & \frac{25\|V\|_2^2}{8\po 3\|V\|_\infty - 2\rho\pf}\,.
%\end{eqnarray*}
%where we have simply bounded $1/(1+\gamma^2)$ by 1. As a consequence,
\begin{eqnarray*}
\sup_{\lambda>0} \left\{\lambda r - \Lambda(\lambda V)\right\} & \geqslant & \sup_{\lambda\in[0,\lambda_0)} \left\{\lambda r - \frac{\beta  \lambda^2 }{\lambda_0-\lambda} \right\}\\
& = & \frac{\lambda_0 r^2}{\beta \po 1 + \sqrt{1+r/\beta}\pf^2}  \ \geqslant \ \frac{\lambda_0 r^2}{4(\beta+r)} \ = \  h(r)\,.
\end{eqnarray*}
%\[ \ = \ \frac{\lambda_0 r^2}{\beta \po 1 + \sqrt{1+r/\beta}\pf^2}  \ \geqslant \ h(r) \,,\]
On conclut en prenant le supremum en $\lambda\in[0,\lambda_0)$ dans \eqref{EqDeviation}.
\end{proof}

Soit $U$ un sous-ensemble mesurable de $E$ et soit $T_U = \inf\{t\geqslant 0,\ X_t\in U\}$ le premier temps d'atteinte de $U$ par le processus $(X_t)_{t\geqslant 0}$. 
\begin{thm}\label{TheoremHitting}
Sous l'Hypothèse \ref{HypoDMS}, si  $\mu(U)>0$, alors pour tout $\theta< h(\mu(U))$ et toute mesure de probabilité $\nu\ll \mu$,
\begin{eqnarray*}
\mathbb E_\nu \po e^{\theta T_U} \pf & \leqslant & 1+  \sqrt 2 \left\| \frac{\dd \nu}{\dd \mu}\right\|_2  \frac{\theta}{h\po \mu(U)\pf-\theta}\,.
\end{eqnarray*}
\end{thm}
\begin{proof}
En suivant \cite{CattiauxGuillinPAZ} on remarque que pour tout $t\geqslant 0$,
\[\{ T_U \geqslant t\} \subset \left\{\frac1t\int_0^t \1_U(X_s) \dd s = 0\right\}\,.\]
Le Théorème \ref{TheoremDeviation} appliqué à $V=-\1_U$ et $r=\mu(U)$ donne
\begin{eqnarray*}
\mathbb P_\nu \po -\frac1t \int_0^t \1_U(X_s) \dd s +\mu(U) \geqslant \mu(U) \pf  &  \leqslant  &\sqrt 2 \left\| \frac{\dd \nu}{\dd \mu}\right\|_2 e^{-t h\po \mu(U)\pf }\,.
\end{eqnarray*}
On conclut alors par
\[\mathbb E_\nu \po e^{\theta T_U} \pf \ = \ 1 + \int_0^\infty \theta e^{\theta t} \mathbb P_\nu \po T_U\geqslant  t\pf \dd t \ \leqslant \ 1 + \sqrt 2 \left\| \frac{\dd \nu}{\dd \mu}\right\|_2 \int_0^\infty \theta  e^{t\po\theta- h\po \mu(U)\pf\pf } \dd t\,.\]
\end{proof}
 
En particulier, le Théorème \ref{TheoremHitting} implique que, pour tout $U$ avec $\mu(U)>0$,  $W(x):=\mathbb E_x \po e^{\theta T_U} \pf $ est finie pour $\mu$-presque tout $x\in E$, et $W\in L^1(\mu)$. En fait on peut avoir mieux sous hypothèse de régularité. 

\begin{assu}\label{HypoHypoco1}
Le noyau de transition $p_t(x,\dd y)$ du processus admet pour tout $t>0$ et $x\in E$ une densité  $r_t$ par rapport à $\mu$, autrement dit $p_t(x,\dd y) = r_t(x,y)\mu(\dd y)$, qui est telle que $y\mapsto r_t(x,y)$ est dans  $L^2(\mu)$ pour tout $x\in E$ avec $x\mapsto \|r_t(x,\cdot)\|_2$  localement bornée.
\end{assu}

Sous les Hypothèses \ref{HypoDMS} et \ref{HypoHypoco1}, en notant $T'_U = \inf\{ t\geqslant 1,\ X_t\in U\}$, alors $T_U' \geqslant T_U$ et 
\begin{eqnarray}\label{EqTemps}
\mathbb E_x \po e^{\theta T_U}\pf & \leqslant & \mathbb E_x \po e^{\theta T_U'}\pf \ = \ \mathbb E_{p_1(x,\cdot)} \po e^{\theta T_U}\pf\,. 
\end{eqnarray}
Par le Théorème \ref{TheoremHitting} et l'hypothèse sur $x\mapsto \|r_t(x,\cdot)\|_2$,  $W(x)$ est donc finie pour $x\in E$ et $W$ est localement bornée.
 
\begin{assu}\label{HypoHypoco2}
$E$ est une variété $\mathcal C^\infty$ de dimension $d$ et  $L=Y_0+\sum_{i=1}^m Y_i^2 $ pour un certain $m\geqslant 1$ où $Y_0,\dots,Y_m$ sont des champs de vecteurs $\mathcal C^\infty$ bornés dont toutes les dérivées sont bornées et tels que, pour un certain $\alpha>0$ et un certain $N\in\N$,
\begin{eqnarray}\label{eqHormander} 
\forall x\in E\,,\, y\in\R^{d}\,,\qquad \sum_{j=1}^m \pso Y_j(x),y\psf^2 + \sum_{Z\in L_N} \pso Z(x),y\psf^2 & \geqslant & \alpha |y|^2\,,
\end{eqnarray}
où $L_N$ est l'ensemble des crochets de Lie de $Y_0,\dots,Y_m$ de longueur dans $\cco 1,N\ccf$.
\end{assu} 
 
\begin{cor}\label{CorLyap}
Sous les Hypothèses \ref{HypoDMS}, \ref{HypoHypoco1} et \ref{HypoHypoco2}, soit $U$ un sous-ensemble mesurable compact de $E$ avec $\mu(U)>0$ et $\theta \in(0, h(\mu(U)))$. Soit $W(x) = \mathbb E_x(e^{\theta T_U})$. Alors $W$ est dans $\mathcal C^\infty(\R^d) \cap L^1(\mu)$ et est solution du problème de Dirichlet
 \[LW+\theta W =0 \qquad\text{on }U^c\,,\qquad\qquad W=1\qquad \text{on }U\,.\]
\end{cor} 
 
En particulier, dans ce cas, $W$ est une fonction de  Lyapunov pour $L$ au sens où 
\[LW \ \leqslant \ -\theta W + C\,,\qquad C:= \sup\{LW(x)+\theta W(x)\, :\, x\in U\}\,.\] 
 
\begin{proof}
La preuve est basée sur \cite[Theorem 5.14]{Cattiauxhormander} et similaire à la preuve de (H2)$\Rightarrow$(H1) dans \cite{CattiauxGuillinPAZ}. 
\end{proof} 

L'hypo-ellipticité forte requise par l'Hypothèse \ref{HypoHypoco2}, qui est déjà supposée dans l'étude du case réversible dans \cite{CattiauxGuillinPAZ}, est assez restrictive, en particulier lorsque $E$ n'est pas compact (ce qui est typiquement le cas où le Corollaire~\ref{CorLyap} est intéressant). Il devrait être possible de montrer que $W$ est une fonction de Lyapunov pour $L$ sous des conditions plus faibles, mais cette question dépasse le cadre élémentaire de la présente note. 

%Since we are possibly motivated by non-diffusion processes,  we  will rather establish that $W$ is a Lyapunov function for $P_t$ for some $t>0$, i.e. that
%\[P_t W \ \leqslant \ e^{-\theta t} W(x) + C\]
%for some $C>0$. This integrated in time condition can then be
%
%[hypothese de noyau de transition dans $L^2$...]

\section{Quelques exemples}

Le cobaye classique pour l'hypococercivité est la diffusion de Langevin (ou de Fokker-Planck cinétique) sur  $\R^d\times\R^d$ de générateur
\[L f \ = \ v\cdot\na_x f - \po \na U(x)+ v\pf \cdot \na_v f + \Delta_v f\]
pour un certain $U\in\mathcal C^2(\R^d)$. Supposons que $\int_{\R^d} e^{-U(x)}\dd x<+\infty$ et notons $\mu$ la mesure de probabilité sur $\mathbb R^d\times\R^d$ de densité par rapport à la mesure de Lebesgue proportionnelle à  $e^{-U(x)-|v|^2/2}$. Supposons que
\[\liminf_{|x|\rightarrow +\infty} \po |\na U(x)|^2-2\Delta U(x)\pf \ > \ 0\]
et qu'il existe $c_1>0$, $c_2\in[0,1)$ et $c_3>0$ tels que, sur $\R^d$,
\[\Delta U \leqslant c_1 + \frac{c_2}{2}|\na U|^2\,,\qquad |\na^2 U| \leqslant c_3 \po 1 + |\na U|\pf\,.\]
Alors, en suivant la preuve de \cite[Theorem 10]{DMS2011} (plus précisément la construction de l'opérateur borné $ A$), l'Hypothèse~ \ref{HypoDMS} est satisfaite de sorte que les Théorèmes~\ref{TheoremDeviation} et \ref{TheoremHitting} s'appliquent. En comparaison, l'approche à la Meyn-Tweedie, qui fournit elle aussi l'intégrabilité exponentielle des temps d'atteinte, a été appliquée à la diffusion de Langevin sous diverses conditions sur    $U$. Par exemple, du fait de la difficulté à construire une fonction de Lyapunov convenable, sept conditions techniques sont requises dans \cite[Hypothesis 1.1]{Talay} sur $U$, qui mettent en jeu une fonction $R$ dont l'existence est alors établie dans quelques cas particuliers. Les conditions de \cite{DMS2011} sont plus générales et plus simples à vérifier.

Comme mentionné dans l'introduction, la comparaison entre  \cite{Andrieu} et\cite{DurmusGuillinMonmarche} pour le \emph{Bouncy Particle process} est à nouveau à l'avantage de la méthode DMS. D'autres exemples où la méthode DMS s'applique avec succès (et donc où l'Hypothèse~\ref{HypoDMS} est satisfaite) sont présentés dans \cite{ReyBellet} et les références qui y figurent.

%Note that Corollary \ref{CorLyap} does not apply in the case of the Langevin dynamics since the vector field $Y_0$ is not bounded. It could be possible to apply the DMS method for a Langevin process with unit speed in the spirit of \cite{BaudoinTardif}, in which case Assumptions~\ref{HypoHypoco1} and \ref{HypoHypoco2} would hold provided the confining forces are bounded with bounded derivatives, but we will rather give an other example where Corollary \ref{CorLyap} applies. 

\bigskip

Autre exemple, considérons la diffusion fortement auto-intéragissante étudiée dans  \cite{BenaimGauthier}, qui est le processus $(X_t)_{t\geqslant 0}$ sur le $\T^d$ (avec $\T = \R/(2\pi\Z)$) solution de
\[\dd X_t \ = \ \dd B_t - \int_0^t \nabla_{x_1} V(X_t,X_s) \dd s\]
où $(x_1,x_2)\in \T^d\times \T^d \mapsto V(x_1,x_2) \in \R$ est un potentiel $\mathcal C^\infty$ et $(B_t)_{t\geqslant 0}$ est un mouvement Brownien standard de dimension $d$. Sous l'hypothèse supplémentaire que $V$ peut être décomposée en
\[V(x_1,x_2) \ = \ \sum_{j=1}^n a_j e_j(x_1)e_j(x_2)\]
où $n\in\N_*$ et, pour tout $j\in\cco 1,n\ccf$, $e_j$ est une fonction propre du laplacien (avec $\pso e_j,e_k\psf = 0$ si $j\neq k$) et $a_j\in\R$, alors le processus peut s'étendre en un processus de Markov de dimension finite en posant $U_{j,t}  =  \int_0^t e_j(X_s)\dd s$. En effet, $(X,U_1,\dots,U_j)\in\T^d\times\R^n$ est solution du système d'équations différentielles stochastiques
\begin{eqnarray*}
\dd X_t & = & \dd B_t - \sum_{j=1}^n a_j \nabla e_j(X_t) U_{j,t} \dd t \\
\forall j\in\cco 1,n\ccf\,,\qquad \dd U_{j,t} & = & e_j(X_t)\dd t\,.
\end{eqnarray*}
L'exemple de base est $V(x_1,x_2) = \cos(x_1-x_2) = \cos(x_1)\cos(x_2)+\sin(x_1)\sin(x_2)$ quand $d=1$, auquel cas le système devient
\begin{eqnarray*}
\dd X_t &= & \dd B_t + \sin(X_t) U_{1,t} \dd t - \cos(X_t) U_{2,t} \dd t\\
\dd U_{1,t} & =& \cos(X_t)\dd t\\
\dd U_{2,t} & = & \sin(X_t)\dd t\,.
\end{eqnarray*}

Dans le cas général, sous réserve que $a_j>0$ pour tout $j\in\cco 1,n\ccf$, le processus admet une unique mesure de probabilité invariante
\[\mu(\dd x,\dd u_1,\dots,\dd u_n) \ \propto \ \exp\po - \frac12\sum_{j=1}^n a_j|\lambda_j|^2 u_j^2\pf \dd x\dd u_1\dots\dd u_n \]
où $\lambda_j$ est la valeur propre du laplacien associée à $e_j$. Il est démontré dans \cite[Section 5]{BenaimGauthier} que la méthode DMS s'applique, et donc que l'Hypothèse~\ref{HypoDMS} est satisfaite. En comparaison, nous sommes au courant de tentatives non publiées de construction d'une fonction de Lyapunov pour le système $(X,U_1,\dots,U_n)$, qui étaient fructueuses mais très délicate dans le cas particulier $V(x_1,x_2) = \cos(x_1-x_2)$ et qui n'ont pu être étendues au cas général (ce qui souligne encore le fait que la construction d'une fonction de Lyapunov est une tâche \emph{ad hoc} qui doit être répétée à chaque nouveau système).
%The hypoellipticity condition \eqref{eqHormander} is straightforwardly deduced from \cite[Section 4]{BenaimGauthier}. The generator of the process is 
%\[L f \ = \ \Delta_x f + \sum_{j=1}^n \po a_j u_j \na e_j(x) \na_x f + e_j(x) \na_{u_j} f\pf\,.\]

\bigskip

Enfin, remarquons que l'Hypothèse~\ref{HypoHypoco2} n'est satisfaite dans aucun des exemples ci-dessus (en particulier pour les diffusions de Langevin et auto-interagissantes, $Y_0$ n'est pas borné).  Donnons donc maintenant un exemple pour s'assurer que le cadre d'application du Corollaire~\ref{CorLyap} n'est pas vide. Considérons la diffusion $(X_t,U_t)_{t\geqslant 0}$ sur $E=\R\times\T$ solution de
\[\left\{\begin{array}{rcl}
\dd X_t & = & \cos(U_t)\dd t\\
\dd U_t & = & -  V'(X_t) \sin(U_t)\dd t + \dd B_t\,,
\end{array}\right.\]
avec $V(x) = x^2/\sqrt{1+x^2}$. C'est une diffusion hypo-elliptique, non elliptique sur un espace non-compact. Le générateur associé est   $L=Y_0+Y_1^2$ avec 
\[Y_1 = \partial_u\,,\qquad Y_0 = \cos(u)\partial_x -V'(x)\sin(u)\partial_u\,,\]
qui sont des champs de vecteurs $\mathcal C^\infty$, bornés et dont toutes les dérivées sont bornées. Posons $Z_1 = \partial_u$,
\[\begin{array}{rcccl}
Z_2 &=& [Y_1,Y_0]& =& -\sin(u)\partial_x -V'(x)\cos(u)\partial_u\\
Z_3&=&[Y_1,Z_2] &=& -\cos(u)\partial_x +V'(x)\sin(u)\partial_u\,.
\end{array}\] 
Alors pour tout $f\in\mathcal C^\infty(E)$ et $(x,u)\in E$,
\[\po Z_2 f(x,u)\pf^2 \ \geqslant \ \frac12 \sin^2(u) \po \partial_x f(x,u)\pf^2 - \|V'\|_\infty^2 \cos^2(u)\po \partial_u f(x,u)\pf^2\,,\]
et de même pour $Z_3$. En posant $a_1 = \|V'\|_\infty^2+1/2$ et $a_2=a_3=1$, on obtient
\[\sum_{j=1}^3 a_j \po Z_j f(x,u)\pf^2 \ \geqslant \ \frac12 |\na f(x,u)|^2\]
pour tout $f\in\mathcal C^\infty(E)$ et tout $(x,u)\in E$. L'Hypothèse~\ref{HypoHypoco2} est donc satisfaite.

Dans un second temps, montrons que la méthode DMS s'applique ici, c'est-à-dire que les conditions $(H_1)-(H_4)$ de \cite{DMS2011} sont remplies. Soit $\mu$ la mesure de probabilités sur $E$ de densité par rapport à la mesure de Lebesgue proportionnelle à $\exp(-V(x))$. Notons $Q=\partial_u^2$ et $T = Y_0$. Alors, par intégration par partie, on voit que $Q^* = Q$ et $T^*=-T$, où $G^*$ désigne le dual dans $L^2(\mu)$ d'un opérateur $G$.  En particulier, pour toutl $f\in \mathcal C^\infty_c(E)$,
\[\int_{E} L f \dd \mu \ = \ \int_{E} f (Q^*+T^*)\1 \dd \mu \ = \ \int_{E} f (Q-T)\1 \dd \mu \ = \ 0\,,  \]
ce qui montre que $\mu$ est invariante pour $L$. L'inégalité de Poincar\'e sur $\T$ donne
\[ \int_{E} (\partial_u f(x,u))^2 e^{-V(x)}\dd x \dd u \ \geqslant \ 4\pi^2 \int_E \po f(x,u) - \int_{\T} f(x,v) \dd v\pf^2 e^{-V(x)}\dd x \dd u\]
pour tout $f\in \mathcal C^\infty_c(E)$ (et donc, par densité, pour tout $f\in D(Q)$), ce qui est exactement la condition de coercivité microscopique $(H_1)$ de \cite{DMS2011}. Notons $\Pi f(x,u) = \int_\T f(x,v)\dd v$. Alors
\[T\Pi f(x,u) \ = \ \cos(u)\partial_x \int_\T f(x,v)\dd v\,,\]
et en particulier, puisque $\int_T\cos(u)\dd u = 0$, $T\Pi T = 0$, qui est $(H_3)$ dans \cite{DMS2011}. De plus, en intégrant derechef par partie, 
\[(T\Pi)^* f(x,u) \ = \ \int_\T \cos(v) \po V'(x)f(x,v) - \partial_x f(x,v)\pf\dd v\,,\]
de sorte que
\begin{eqnarray*}
(T\Pi)^*T\Pi f(x,u) & = & \int_\T \cos(v)^2 \po V'(x)\partial_x - \partial_x^2\pf \int f(x,w)\dd w \dd v\\
& =& \frac12 \po V'(x)\partial_x - \partial_x^2\pf \Pi f(x,u)\,.
\end{eqnarray*} 
En consequence,
\[\int_E \po T\Pi f\pf^2 \dd \mu \ = \ \frac12 \int_\R \po \partial_x \Pi f\pf^2 e^{-V} \ \geqslant \ \frac12C_P \int \po \Pi f -\mu f\pf^2 \dd \mu\,,\]
où $C_P$ est la constante de Poincaré de la mesure $e^{-V}$. C'est la condition de coercivité macroscopique $(H_2)$ de \cite{DMS2011}. Au vu de l'expression de $(T\Pi)^*T\Pi $, la dernière condition $(H_4)$ de \cite{DMS2011} suit de \cite[Lemma 4]{DMS2011}. En consequence, l'Hypothèse~ \ref{HypoDMS} est satisfaite.

Finalement, démontrons que l'Hypothèse~\ref{HypoHypoco1} est satisfaite. En notant $\tilde U_t = U_t+\pi$, remarquons que
\[\left\{\begin{array}{rcl}
\dd X_t & = & -\cos(\tilde U_t)\dd t\\
\dd \tilde U_t & = &  V'(X_t) \sin(\tilde U_t)\dd t + \dd B_t\,.
\end{array}\right.\]
En d'autres termes, $(X_t,U_t+\pi)_{t\geqslant 0}$ est un processus de Markov de générateur $L^*$ (ce qui n'est pas sans rappeler le changement de variable $w=-v$ pour la diffusion de Langevin). En notant $p_t$ et $p_t^*$ les noyaux de transition de $L$ et $L^*$, on a donc, pour tout  $ (x,u),(y,w)\in E$,
\[p_t\po (y,w),(x,u)\pf  \ = \ p_t^*\po (x,u),(y,w)\pf \  = \ p_t\po (x,u-\pi),(y,w-\pi)\pf \,.\]
Puisque $\mu$ est invariante par la transformation $(x,u)\mapsto (x,u-\pi)$, les densités $r_t=p_t/\mu$ et $r_t^*=p_t/\mu$ satisfont la même relation. D'après \cite[Theorem 1.5]{CattiauxhormanderI}, l'Hypothèse~\ref{HypoHypoco2} implique que $p_t((x,u),(y,w)) \leqslant C(1\wedge t)^{-M}$ pour certains $C,M>0$ uniformément en $(x,u),(y,w)\in E$. On borne alors
\begin{eqnarray*}
\|r_t\po (x,u),\cdot\pf\|_2^2 & = & \int_E r_t\po (x,u),(y,w)\pf r_t\po (y,w+\pi),(x,u+\pi)\pf e^{-V(y)} \dd y \dd w  \\
& \leqslant & \frac{Ce^{V(x)}}{\po 1\wedge t\pf^M}  \int_E r_t\po (x,u),(y,w)\pf  e^{-V(y)} \dd y \dd w \ = \ \frac{Ce^{V(x)}}{\po 1\wedge t\pf^M} \,,
\end{eqnarray*}
ce qui établit l'Hypothèse~\ref{HypoHypoco1}.

En conclusion, Corollaire~\ref{CorLyap} s'applique pour le processus $(X_t,U_t)_{t\geqslant 0}$.De plus, le Théorème~\ref{TheoremHitting}, l'inégalité~\eqref{EqTemps} et la borne quantitative sur $\|r_t\po (x,u),\cdot\pf\|_2$ donnent
\[  \mathbb E_{(x,u)} \po e^{\theta T_U} \pf \leqslant \ C e^{\frac12 V(x)}\]
pour un certain $C>0$ pour tout $x\in E$ et tout ensemble $U$ avec $\mu(U)>0$ et tout $\theta<h\po \mu(U)\pf$.

\bigskip

\noindent\textbf{Support financier.} Ce travail a été supporté par le project EFI ANR-17-CE40-0030 de l'Agence Nationale de Recherche Française.

\bibliographystyle{plain}
\bibliography{biblio}
\end{document}

\maketitle

\abstract{We establish that, for a Markov semi-group, $L^2$ hypocoercivity, i.e. contractivity for a modified $L^2$ norm, implies quantitative deviation bounds for additive functionals of the associated Markov process and exponential integrability of the hitting time of sets with positive measure. Moreover, in the case of diffusion processes and under a strong hypoellipticity assumption, we prove that $L^2$ hypocoercivity implies the existence of a Lyapunov function for the generator. This is an english translation of the original article in french.}

\section{Introduction}

%Motivation : lien entre hypocoercivité (en particulier à la Dolbeault-Mouhot-Schmeiser \cite{DMS2011}) et Meyn-Tweedie/Foster-Lyapunov \cite{HairerMattingly2008}. C'est le genre de question considérée dans le cas réversible dans \cite{Cattiaux2008,CattiauxGuillinPAZ}. On se rend compte en voulant l'adapter que, via \cite{CattiauxGuillin}, c'est basé sur les déviations de \cite{Wu}. 

This note is primarily motivated by the comparison between two methods for obtaining quantitative rates of convergence for ergodic Markov processes: functional inequalities and entropy dissipation on the one hand, and the classical Meyn-Tweedie approach on the other hand, based on a Foster-Lyapunov criterion together with a local Doeblin or coupling condition. The link between these two kind of arguments is studied by Cattiaux, Guillin and their co-authors in a series of works \cite{Cattiaux2008,CattiauxGuillin3,CattiauxGuillin2,CattiauxGuillinPAZ}, but mainly in the case of reversible Markov processes, with an invariant measure that typically satisfies a Poincar\'e (or similar) inequality (with respect to the Dirichlet form associated with the process). Nevertheless, over the last decade, hypocoercivity methods of modified entropies have proven to be able to handle non-reversible, non-elliptic and non-diffusive processes with functional inequality arguments. Here we will focus on the Dolbeault-Mouhot-Schmeiser (DMS) method \cite{DMS2011} for $L^2$ hypocoercivity. A very appealing feature of the latter is that it provides a general construction for a modified $L^2$-norm, independently from the process. This is quite different in a Meyn-Tweedie approach where a Lyapunov function has to be constructed based on the dynamics of the process, which can be tricky for degenerate processes. An example is provided by the comparison of the two works \cite{Andrieu} and \cite{DurmusGuillinMonmarche} that both study the so-called Bouncy Particle process, which is a kinetic piecewise deterministic Markov process, respectively with the DMS and the Meyn-Tweedie method. With the Meyn-Tweedie approach, the construction of the Lyapunov function is quite intricate and this leads to not-so-nice assumptions on the log-density of the equilibrium, while with the DMS method the construction is standard and the conditions on the log-density are more general and quite simpler. A  natural question is thus: provided that (hypocoercive) exponential decay holds in $L^2$, does it exists a Lyapunov function in the sense of Meyn-Tweedie ? A positive answer is provided in \cite{CattiauxGuillinPAZ} in the case of elliptic reversible diffusions for which the exponential decay holds with the usual $L^2$ norm. The result is based on the exponential integrability of the hitting times of the process, itself obtained from the deviation bounds on additive functionals of the process derived in \cite{Wu,CattiauxGuillin}. These two results are interesting by themselves; in fact, in some cases, getting estimates on the hitting times is the main question and the Lyapunov function is just an intermediate tool to get them. Our main result is that, if the DMS method applies, then these two results also hold. From the exponential integrability of hitting times, in the case of strongly hypoelliptic diffusions, we construct a Lyapunov function for the generator, which answers our initial question.  In fact, as we became aware after the redaction of this note, the deviation bounds have been established in the recent preprint \cite{ReyBellet}. To the best of our knowledge, the other results are new.

\section{Results and proofs}

Consider $(X_t)_{t\geqslant 0}$  a continuous-time conservative Markov process on a Polish space $E$, with some invariant probability measure $\mu$. Denote $(P_t)_{t\geqslant 0}$ the associated Markov semi-group on $L^2(\mu)$  and $L$. We suppose that $L$ is closed and that its domain $D(L)$ is dense in $L^2(\mu)$. Denote $\|\cdot\|_2$ and $\langle\cdot\rangle$ the usual norm and scalar product on $L^2(\mu)$. Our main assumption is the following: 

\begin{assu}\label{HypoDMS}
There exists $\rho>0$ and a  symmetric bounded linear operator $S$ on $L^2(\mu)$ such that for all $f\in D(L)$,
\begin{eqnarray*}
\| S f\|_2 & \leqslant & \frac12 \| f-\mu f\|_2 \\
\langle f,Lf + SLf\rangle  & \leqslant & - \rho \| f-\mu f\|_2^2\,.
\end{eqnarray*}
\end{assu}
The existence of such an $S$, which will be used in the present note as a black box assumption, is in fact the main tool of the DMS method. More precisely, considering $\varepsilon\in(0,1)$ and the operator $A$ that are used to define the modified entropy $H$ in  \cite{DMS2011}, we set $S=\varepsilon(\rm{Id}-\mu)( A +  A^*)(\rm{Id}-\mu)/2$. It is straightforward to check that, under the assumption $\mathrm{(H_1)}-\mathrm{(H_4)}$ of \cite{DMS2011}, this operator $S$ satisfies the conditions of Assumption \ref{HypoDMS}. Once such an operator $S$ is obtained, the DMS method to obtain hypocoercive decay in $L^2(\mu)$ is the following: denote $B=$Id$+S$ and consider the scalar product and Hilbertian norm
\[\langle f,g\rangle_B \ = \ \langle f, Bg\rangle\,,\qquad \|f\|_B = \sqrt{\langle f,f\rangle_B}\,. \]
The latter is equivalent to the usual norm on $L^2(\mu)$, more precisely $1/2\|f\|_2^2 \leqslant \| f\|_B^2 \leqslant 3/2\| f\|_2^2$.  Assumption~\ref{HypoDMS} implies that for all $f\in D(L)$ with $\mu f = 0$,
\[\pso f,Lf\psf_B \ \leqslant \ -\rho \|f\|_2^2 \ \leqslant \ - \frac{2\rho}{3} \|f\|_B^2\,.\]
In other words, $(e^{2\rho t/3}P_t)_{t\geqslant 0}$ is dissipative on $\{f\in L^2(\mu),\ \mu f=0\}$ endowed with the scalar product $\pso\cdot\psf_B$, so that by the   Lumer-Philips Theorem \cite[Chapter IX, p.250]{Yosida}, for all $f$ with $\mu f= 0$,
\[\|P_t f\|_B \ \leqslant\ e^{-2\rho t/3}\| f\|_B\,. \]
This gives an hypocoercive decay in the usual norm of $L^2(\mu)$:
\[\| P_t f - \mu f\|_2 \ \leqslant \ \sqrt 2 \|P_t f -\mu f\|_B \ \leqslant \ \sqrt 2e^{-2\rho t/3}\|f-\mu f\|_B \ \leqslant \  \sqrt 3 e^{-2\rho t/3}\|f-\mu f\|_2\,.\]

In fact, given a bounded measurable function $V$ on $E$, the same argument works for the Feynman-Kac semigroup $(P_t^V)_{t\geqslant 0}$ on $L^2(\mu)$ given by
\[P_t^V f(x) \ = \ \mathbb E_x \po f(X_t) e^{\int_0^t V(X_s)\dd s}\pf\,,\]
where the subscript $x$ denotes the starting point of the process.  Indeed, denoting
\[\Lambda(V) \ = \ \sup\left\{ \pso f, (L+V)f \psf_B \,:\, f\in D(L)\ ,\ \|f\|_B = 1 \right\}\,,\]
we get that $(e^{-t\Lambda(V)}P_t^V)_{t\geqslant 0}$ is dissipative on $(L^2(\mu),\pso\cdot\psf_B)$ and thus, for all $f\in L^2(\mu)$,
\begin{eqnarray}\label{EqPtV}
\| P_t^V f\|_B & \leqslant  & e^{t\Lambda(V)} \|f\|_B\,.
\end{eqnarray}
From this, following \cite{Wu} and \cite{CattiauxGuillin}, we can establish the following deviation bounds:

\begin{thm}\label{TheoremDeviation}
Under Assumption \ref{HypoDMS}, let $\nu \ll \mu$ be a probability measure on $E$ and $V$ be a bounded measurable function on $E$. Then, for all $t\geqslant 0$ and all $r\geqslant 0$, 
\begin{eqnarray*}
\mathbb P_\nu \po \frac1t \int_0^t V(X_s) \dd s -\mu V \geqslant r \pf & \leqslant &  \sqrt 2 \left\| \frac{\dd \nu}{\dd \mu}\right\|_2 e^{-t h(r)} \,,
\end{eqnarray*}
where 
\[h(r) = \frac{\rho r^2}{25\|V\|_2^2 + 6 \| V\|_\infty r} \,. \]
\end{thm}
%
%More explicitly, denoting $\lambda_0 = \rho/(3\|V\|_\infty)$ and $r_0 = 25\|V\|_2^2/(2\|V\|_\infty)$, 
%\[h(r) \ = \ \frac{\rho (r\wedge r_0)^2}{25\|V-\mu V\|_2^2} + \lambda_0 \max(r-r_0,0)\,.\]
%\left\{\begin{array}{ll}\frac{\rho r^2}{25\|V-\mu V\|_2^2} & \text{if }r\leqslant r_0 \\& \\\lambda_0 r - \frac{25}{4\rho}\lambda_0^2\|V-\mu V\|_2^2  & \text{otherwise.}\end{array}\right. 
%and in all cases
%\[h(r) \ \geqslant \ \sup_{\lambda \in[0,\lambda_0]} \left\{ \lambda (r\wedge r_0) - \frac{25\lambda^2\|V-\mu V\|_2^2}{4\rho}\right\} \ = \ \frac{\rho (r\wedge r_0)^2}{25\|V-\mu V\|_2^2}\]

Of course, by changing $V$ to $-V$, this provides a non-asymptotic confidence interval for the empirical mean of bounded functions.

Since the DMS method does not yield sharp estimates, we have prefered a simple expression for $h(r)$ rather than the largest estimate we could obtain.

Although a similar result is already established in \cite{ReyBellet}, since the proof is nice and short, we leave it for completeness.

\begin{proof}
Without loss of generality we suppose that $\mu V = 0$. By the Chebyshev  and Cauchy-Schwarz inequalities and using \eqref{EqPtV},  for all $\nu \ll \mu$ and $\lambda\geqslant 0$,
\begin{eqnarray}
\mathbb P_\nu \po \frac1t \int_0^t V(X_s) \dd s  \geqslant r \pf & \leqslant & e^{-\lambda t r} \mathbb E_\nu \po e^{\lambda\int_0^t V(X_s)\dd s} \pf \notag\\
& =& e^{-\lambda t r} \int_E P_t^{\lambda V} \1 \dd \nu\notag \\
& \leqslant & \sqrt 2 e^{-\lambda t r}  \left\| \frac{\dd \nu}{\dd \mu}\right\|_2 \left\|P_t^{\lambda V}\1\right \|_B\notag\\
& \leqslant & \sqrt 2 \left\| \frac{\dd \nu}{\dd \mu}\right\|_2 e^{-\lambda t r + t \Lambda(\lambda V)} \label{EqDeviation}\,.
\end{eqnarray}
We now bound $\Lambda(\lambda V)$ in the spirit of \cite{CattiauxGuillin}. First, from Assumption~\ref{HypoDMS},  
\begin{eqnarray*}
\Lambda(V) & \leqslant & 2 \sup\left\{ -\rho \| f-\mu f\|^2  +  \pso f, V f \psf_B \,:\, f\in D(L)\ ,\ \|f\|_2 = 1 \right\}
\end{eqnarray*} 
Fix $f\in D(L)$ with $\|f\|_2 = 1$ and let $\gamma \geqslant 0$ be given by $1+\gamma^2 = 1/(\mu f)^2$, so that
\[g \ := \ \frac1\gamma \po \sqrt{1+\gamma^2} f - 1\pf \ = \ \frac{\sqrt{1+\gamma^2}}{\gamma}(f-\mu f)\]
satisfies $\mu g = 0$ and
\[\mu( g^2) \ = \ \frac{1+\gamma^2}{\gamma^2}  \po \mu(f^2) - (\mu f)^2 \pf \ = \ 1\,.\]
%\begin{eqnarray*}
%\| g\|_B^2 & = & \pso g,g\psf + \pso g,Ag\psf\\
%& = & \frac{1+\gamma^2}{\gamma^2} \po \pso f,f\psf - \po \mu f\pf^2 +   \pso f,A f\psf\pf\\
%& = & \frac{1+\gamma^2}{\gamma^2} \po  \| f\|_B^2 - \frac1{1+\gamma^2}\pf \ = \ 1\,.
%\end{eqnarray*}
Moreover,
\begin{eqnarray*}
\pso f, V f \psf_B & = & \int V f^2 \dd \mu + \pso A (f-\mu f),(Vf-\mu(Vf))\psf\\
 &  = & \int V f^2 \dd \mu + \pso A (f-\mu f),V(f-\mu f)\psf +   \pso A (f-\mu f),(V-\mu V)\psf \mu f \\
 & = & \frac{1}{1+\gamma^2} \int V \po \gamma^2 g^2 + 2 \gamma g +1 \pf \dd \mu + \frac{\gamma^2}{1+\gamma^2} \pso A g,V g\psf + \frac{\gamma}{1+\gamma^2}\pso A g,V \psf\\
 & \leqslant& \frac{3\gamma^2}{2(1+\gamma^2)}\| V\|_\infty + \frac{5\gamma}{2(1+\gamma^2)}\| V\|_2 \,,
\end{eqnarray*}
where we used the Cauchy-Schwarz inequality and the fact that $\|Ag\|\leqslant 1/2$. As a consequence,
\[\pso f, (L+V) f \psf_B \ \leqslant\  \frac{\gamma}{1+\gamma^2} \po -\gamma\rho + \frac{3\gamma}{2}\| V\|_\infty + \frac52\| V\|_2\pf\,,\]
and for all $\lambda\geqslant 0$,
\begin{eqnarray*}
 \Lambda(V) & \leqslant & 2\sup_{\gamma\geqslant 0} \left\{ \frac{\gamma}{1+\gamma^2} \po -\gamma\rho + \frac{3\gamma}{2}\| V\|_\infty + \frac52\| V\|_2\pf \right\}\\
&  \leqslant & 2 \sup_{\gamma\geqslant 0} \left\{  \gamma \po -\gamma\rho + \frac{3\gamma}{2}\| V\|_\infty + \frac52\| V\|_2\pf \right\}\\
& = & \frac{25\|V\|_2^2}{4\rho-6\|V\|_\infty} \qquad \text{if }3\|V\|_\infty < 2\rho\,,\qquad +\infty\qquad\text{otherwise.}
\end{eqnarray*}
In particular, denoting  
\[\lambda_0 := 2\rho/(3\|V\|_\infty)\,,\qquad \beta = \frac{25 \|V\|_2^2}{6 \|V\|_\infty } \,,\]
we get
%\begin{eqnarray*}
%\frac12 \Lambda(\lambda V) & \leqslant & \sup_{\gamma\geqslant 0} \left\{ \frac{\gamma}{1+\gamma^2} \po -\frac12\gamma\rho  + \frac52 \lambda\| V\|_2\pf \right\} \ \leqslant \ \frac{25\lambda^2\|V\|_2^2}{4\rho}\,,
%%& \leqslant & \sup_{\gamma\geqslant 0} \left\{ \gamma \po -\gamma\rho + \frac{3\gamma}{2}\| V\|_\infty + \frac52\| V\|_2\pf \right\} \\
%%&= & \frac{25\|V\|_2^2}{8\po 3\|V\|_\infty - 2\rho\pf}\,.
%\end{eqnarray*}
%where we have simply bounded $1/(1+\gamma^2)$ by 1. As a consequence,
\begin{eqnarray*}
\sup_{\lambda>0} \left\{\lambda r - \Lambda(\lambda V)\right\} & \geqslant & \sup_{\lambda\in[0,\lambda_0)} \left\{\lambda r - \frac{\beta  \lambda^2 }{\lambda_0-\lambda} \right\}\\
& = & \frac{\lambda_0 r^2}{\beta \po 1 + \sqrt{1+r/\beta}\pf^2}  \ \geqslant \ \frac{\lambda_0 r^2}{4(\beta+r)} \ = \  h(r)\,.
\end{eqnarray*}
%\[ \ = \ \frac{\lambda_0 r^2}{\beta \po 1 + \sqrt{1+r/\beta}\pf^2}  \ \geqslant \ h(r) \,,\]
Taking the supremum over $\lambda\in[0,\lambda_0)$ in \eqref{EqDeviation} concludes.
\end{proof}

Let $U$ be a measurable subset of $E$ and $T_U = \inf\{t\geqslant 0,\ X_t\in U\}$ be the first hitting time of $U$ of the process $(X_t)_{t\geqslant 0}$. 
\begin{thm}\label{TheoremHitting}
Under Assumption \ref{HypoDMS}, if $\mu(U)>0$, then for all $\theta< h(\mu(U))$ and all probability measure $\nu\ll \mu$,
\begin{eqnarray*}
\mathbb E_\nu \po e^{\theta T_U} \pf & \leqslant & 1+  \sqrt 2 \left\| \frac{\dd \nu}{\dd \mu}\right\|_2  \frac{\theta}{h\po \mu(U)\pf-\theta}\,.
\end{eqnarray*}
\end{thm}
\begin{proof}
Following \cite{CattiauxGuillinPAZ} we remark that for all $t\geqslant 0$,
\[\{ T_U \geqslant t\} \subset \left\{\frac1t\int_0^t \1_U(X_s) \dd s = 0\right\}\,.\]
Theorem \ref{TheoremDeviation} applied with $V=-\1_U$ and $r=\mu(U)$ yields
\begin{eqnarray*}
\mathbb P_\nu \po -\frac1t \int_0^t \1_U(X_s) \dd s +\mu(U) \geqslant \mu(U) \pf  &  \leqslant  &\sqrt 2 \left\| \frac{\dd \nu}{\dd \mu}\right\|_2 e^{-t h\po \mu(U)\pf }\,.
\end{eqnarray*}
Then, conclusion follows from
\[\mathbb E_\nu \po e^{\theta T_U} \pf \ = \ 1 + \int_0^\infty \theta e^{\theta t} \mathbb P_\nu \po T_U\geqslant  t\pf \dd t \ \leqslant \ 1 + \sqrt 2 \left\| \frac{\dd \nu}{\dd \mu}\right\|_2 \int_0^\infty \theta  e^{t\po\theta- h\po \mu(U)\pf\pf } \dd t\,.\]
\end{proof}
 
In particular, Theorem \ref{TheoremHitting} implies that, for all $U$ with $\mu(U)>0$,  $W(x):=\mathbb E_x \po e^{\theta T_U} \pf $ is finite for $\mu$-almost every $x\in E$, and $W\in L^1(\mu)$. In fact we can have better under some regularity assumption.

\begin{assu}\label{HypoHypoco1}
The transition kernel $p_t(x,\dd y)$ of the process admits for all $t>0$ and $x\in E$ a    density $r_t$ with respect to $\mu$, i.e. $p_t(x,\dd y) = r_t(x,y)\mu(\dd y)$, such that $y\mapsto r_t(x,y)$ is in $L^2(\mu)$ for all $x\in E$ with $x\mapsto \|r_t(x,\cdot)\|_2$  locally bounded.
\end{assu}

Under Assumptions \ref{HypoDMS} and \ref{HypoHypoco1}, denoting $T'_U = \inf\{ t\geqslant 1,\ X_t\in U\}$, then $T_U' \geqslant T_U$ and 
\begin{eqnarray}\label{EqTemps}
\mathbb E_x \po e^{\theta T_U}\pf & \leqslant & \mathbb E_x \po e^{\theta T_U'}\pf \ = \ \mathbb E_{p_1(x,\cdot)} \po e^{\theta T_U}\pf\,. 
\end{eqnarray}
 From Theorem \ref{TheoremHitting} and the assumption on $x\mapsto \|r_t(x,\cdot)\|_2$,  $W(x)$ is then finite for all $x\in E$ and $W$ is locally bounded.
 
\begin{assu}\label{HypoHypoco2}
$E$ is a $d$-dimensional $\mathcal C^\infty$ manifold and  $L=Y_0+\sum_{i=1}^m Y_i^2 $ for some $m\geqslant 1$ where $Y_0,\dots,Y_m$ are bounded $\mathcal C^\infty$ vector fields with all their derivative bounded, and such that for some $\alpha>0$ and $N\in\N$,
\begin{eqnarray}\label{eqHormander} 
\forall x\in E\,,\, y\in\R^{d}\,,\qquad \sum_{j=1}^m \pso Y_j(x),y\psf^2 + \sum_{Z\in L_N} \pso Z(x),y\psf^2 & \geqslant & \alpha |y|^2\,,
\end{eqnarray}
where $L_N$ denotes  the set of Lie brackets of $Y_0,\dots,Y_m$ with length in $\cco 1,N\ccf$.
\end{assu} 
 
\begin{cor}\label{CorLyap}
Under Assumptions \ref{HypoDMS}, \ref{HypoHypoco1} and \ref{HypoHypoco2}, let $U$ be a compact measurable subset of $E$ with $\mu(U)>0$ and $\theta \in(0, h(\mu(U)))$. Set $W(x) = \mathbb E_x(e^{\theta T_U})$. Then $W$ is in $\in\mathcal C^\infty(\R^d) \cap L^1(\mu)$ and solves the Dirichlet boundary problem
 \[LW+\theta W =0 \qquad\text{on }U^c\,,\qquad\qquad W=1\qquad \text{on }U\,.\]
\end{cor} 
 
In particular, in that case, $W$ is a  Lyapunov function for $L$ in the sense that 
\[LW \ \leqslant \ -\theta W + C\,,\qquad C:= \sup\{LW(x)+\theta W(x)\, :\, x\in U\}\,.\] 
 
\begin{proof}
The proof is based on  \cite[Theorem 5.14]{Cattiauxhormander} and similar to the proof of (H2)$\Rightarrow$(H1) in \cite{CattiauxGuillinPAZ}. 
\end{proof} 
 
The strong hypo-ellipticity required by Assumption \ref{HypoHypoco2}, which is already enforced in the reversible case in \cite{CattiauxGuillinPAZ}, is a quite restrictive condition, especially if $E$ is not compact (which is the case where Corollary~\ref{CorLyap} is interesting). It should be possible to prove that $W$ is a Lyapunov function for $L$ under weaker conditions, but this question exceeds the scope of the present elementary note.

%Since we are possibly motivated by non-diffusion processes,  we  will rather establish that $W$ is a Lyapunov function for $P_t$ for some $t>0$, i.e. that
%\[P_t W \ \leqslant \ e^{-\theta t} W(x) + C\]
%for some $C>0$. This integrated in time condition can then be
%
%[hypothese de noyau de transition dans $L^2$...]

\section{A few examples}

The classical benchmark for hypocoercivity is the Langevin (or kinetic Fokker-Planck) diffusion on $\R^d\times\R^d$ whose generator is 
\[L f \ = \ v\cdot\na_x f - \po \na U(x)+ v\pf \cdot \na_v f + \Delta_v f\]
for some $U\in\mathcal C^2(\R^d)$. Assume that $\int_{\R^d} e^{-U(x)}\dd x<+\infty$ and let $\mu$ be the probability measure  with Lebesgue density proportional to $e^{-U(x)-|v|^2/2}$. Assume that
\[\liminf_{|x|\rightarrow +\infty} \po |\na U(x)|^2-2\Delta U(x)\pf \ > \ 0\]
and that there exists $c_1>0$, $c_2\in[0,1)$ and $c_3>0$ such that, on $\R^d$,
\[\Delta U \leqslant c_1 + \frac{c_2}{2}|\na U|^2\,,\qquad |\na^2 U| \leqslant c_3 \po 1 + |\na U|\pf\,.\]
Then, following the proof of \cite[Theorem 10]{DMS2011} (more precisely the construction of the bounded operator $ A$), Assumption \ref{HypoDMS} is satisfied, so that Theorem \ref{TheoremDeviation} and \ref{TheoremHitting} hold. As a comparison, the Meyn-Tweedie approach, that also yields the exponential integrability of the hitting times, has been applied to the Langevin diffusion under various conditions on $U$. For instance, due to the difficulty of constructing a suitable Lyapunov function,   seven technical conditions are required in \cite[Hypothesis 1.1]{Talay} on $U$, that involve a function $R$ whose existence is then checked on various examples. The conditions of \cite{DMS2011} are arguably more general and easier to check.

As mentioned in the introduction, the comparison between \cite{Andrieu} and \cite{DurmusGuillinMonmarche} for the Bouncy Particle process is similarly at the advantage of the DMS method. Other examples where the DMS method is successfully applied (and thus where Assumption \ref{HypoDMS} holds) can be found in \cite{ReyBellet} and references within.

%Note that Corollary \ref{CorLyap} does not apply in the case of the Langevin dynamics since the vector field $Y_0$ is not bounded. It could be possible to apply the DMS method for a Langevin process with unit speed in the spirit of \cite{BaudoinTardif}, in which case Assumptions~\ref{HypoHypoco1} and \ref{HypoHypoco2} would hold provided the confining forces are bounded with bounded derivatives, but we will rather give an other example where Corollary \ref{CorLyap} applies. 

\bigskip

Next, consider the strongly self-interacting diffusion studied in \cite{BenaimGauthier}, which is the process $(X_t)_{t\geqslant 0}$ on the torus $\T^d$ (with $\T = \R/(2\pi\Z)$) that solves
\[\dd X_t \ = \ \dd B_t - \int_0^t \nabla_{x_1} V(X_t,X_s) \dd s\]
where $(x_1,x_2)\in \T^d\times \T^d \mapsto V(x_1,x_2) \in \R$ is a $\mathcal C^\infty$ potential and $(B_t)_{t\geqslant 0}$ is a standard $d$-dimensional Brownian motion. Under the additional assumption that $V$ can be decomposed as
\[V(x_1,x_2) \ = \ \sum_{j=1}^n a_j e_j(x_1)e_j(x_2)\]
where $n\in\N_*$ and, for all $j\in\cco 1,n\ccf$, $e_j$ is an eigenfunction of the Laplacian (with $\pso e_j,e_k\psf = 0$ if $j\neq k$) and $a_j\in\R$, then the process can be extended to a finite-dimensional Markov process by setting $U_{j,t}  =  \int_0^t e_j(X_s)\dd s$. Indeed, $(X,U_1,\dots,U_j)\in\T^d\times\R^n$ then solves the system of stochastic equations
\begin{eqnarray*}
\dd X_t & = & \dd B_t - \sum_{j=1}^n a_j \nabla e_j(X_t) U_{j,t} \dd t \\
\forall j\in\cco 1,n\ccf\,,\qquad \dd U_{j,t} & = & e_j(X_t)\dd t\,.
\end{eqnarray*}
The basic example is $V(x_1,x_2) = \cos(x_1-x_2) = \cos(x_1)\cos(x_2)+\sin(x_1)\sin(x_2)$ when $d=1$, in which case the system reads
\begin{eqnarray*}
\dd X_t &= & \dd B_t + \sin(X_t) U_{1,t} \dd t - \cos(X_t) U_{2,t} \dd t\\
\dd U_{1,t} & =& \cos(X_t)\dd t\\
\dd U_{2,t} & = & \sin(X_t)\dd t\,.
\end{eqnarray*}

In the general case, provided $a_j>0$ for all $j\in\cco 1,n\ccf$, the process admits a unique invariant probability measure
\[\mu(\dd x,\dd u_1,\dots,\dd u_n) \ \propto \ \exp\po - \frac12\sum_{j=1}^n a_j|\lambda_j|^2 u_j^2\pf \dd x\dd u_1\dots\dd u_n \]
where $\lambda_j$ is the eigenvalue of the Laplacian associated to $e_j$. It is proven in \cite[Section 5]{BenaimGauthier} that the DMS method, hence Assumption~\ref{HypoDMS}, holds. As a comparison, we are aware of unpublished attempts to construct a Lyapunov function for the system $(X,U_1,\dots,U_n)$, that were succesful but intricate in the particular case $V(x_1,x_2) = \cos(x_1-x_2)$ and couldn't be extended to the general case (which highlight again the fact that constructing a Lyapunov function is a very \emph{ad hoc} task that should be repeated for every new system).
%The hypoellipticity condition \eqref{eqHormander} is straightforwardly deduced from \cite[Section 4]{BenaimGauthier}. The generator of the process is 
%\[L f \ = \ \Delta_x f + \sum_{j=1}^n \po a_j u_j \na e_j(x) \na_x f + e_j(x) \na_{u_j} f\pf\,.\]

\bigskip

Finally, remark that Assumption \ref{HypoHypoco2} is not satisfied in the examples above (in particular, for the Langevin and the self-interacting diffusions, $Y_0$ is not bounded). We now give an example that ensures that the scope of Corollary \ref{CorLyap} is not empty. Consider the diffusion $(X_t,U_t)_{t\geqslant 0}$ on $E=\R\times\T$ that solves
\[\left\{\begin{array}{rcl}
\dd X_t & = & \cos(U_t)\dd t\\
\dd U_t & = & -  V'(X_t) \sin(U_t)\dd t + \dd B_t\,,
\end{array}\right.\]
with $V(x) = x^2/\sqrt{1+x^2}$. This is a non-elliptic, hypoelliptic diffusion on a non-compact space. The associated generator is $L=Y_0+Y_1^2$ with
\[Y_1 = \partial_u\,,\qquad Y_0 = \cos(u)\partial_x -V'(x)\sin(u)\partial_u\,,\]
which are bounded $\mathcal C^\infty$ vector fields with all their derivative bounded. Consider $Z_1 = \partial_u$,
\[\begin{array}{rcccl}
Z_2 &=& [Y_1,Y_0]& =& -\sin(u)\partial_x -V'(x)\cos(u)\partial_u\\
Z_3&=&[Y_1,Z_2] &=& -\cos(u)\partial_x +V'(x)\sin(u)\partial_u\,.
\end{array}\] 
Then for all $f\in\mathcal C^\infty(E)$ and all $(x,u)\in E$,
\[\po Z_2 f(x,u)\pf^2 \ \geqslant \ \frac12 \sin^2(u) \po \partial_x f(x,u)\pf^2 - \|V'\|_\infty^2 \cos^2(u)\po \partial_u f(x,u)\pf^2\,,\]
and similarly for $Z_3$. Setting $a_1 = \|V'\|_\infty^2+1/2$ and $a_2=a_3=1$, we get that
\[\sum_{j=1}^3 a_j \po Z_j f(x,u)\pf^2 \ \geqslant \ \frac12 |\na f(x,u)|^2\]
for all $f\in\mathcal C^\infty(E)$ and all $(x,u)\in E$, and thus Assumption \ref{HypoHypoco2} holds.

Second, let us check that the DMS method applies here, i.e. that the conditions $(H_1)-(H_4)$ of \cite{DMS2011} are satisfied. Let $\mu$ be the probability measure on $E$ with Lebesgue density proportional to $\exp(-V(x))$. Denote $Q=\partial_u^2$ and $T = Y_0$. Then by integration by part, we see that $Q^* = Q$ and $T^*=-T$, where $G^*$ denotes the dual in $L^2(\mu)$ of an operator $G$. In particular, for all $f\in \mathcal C^\infty_c(E)$,
\[\int_{E} L f \dd \mu \ = \ \int_{E} f (Q^*+T^*)\1 \dd \mu \ = \ \int_{E} f (Q-T)\1 \dd \mu \ = \ 0\,,  \]
which proves $\mu$ is invariant for $L$. The Poincar\'e inequality on $\T$ states that
\[ \int_{E} (\partial_u f(x,u))^2 e^{-V(x)}\dd x \dd u \ \geqslant \ 4\pi^2 \int_E \po f(x,u) - \int_{\T} f(x,v) \dd v\pf^2 e^{-V(x)}\dd x \dd u\]
for all $f\in \mathcal C^\infty_c(E)$ (and thus, by density, for all $f\in D(Q)$), which is exactly the microscopic coercivity Assumption $(H_1)$ of \cite{DMS2011}. Denote $\Pi f(x,u) = \int_\T f(x,v)\dd v$. Then
\[T\Pi f(x,u) \ = \ \cos(u)\partial_x \int_\T f(x,v)\dd v\,,\]
and in particular, since $\int_T\cos(u)\dd u = 0$, $T\Pi T = 0$, which is $(H_3)$ of \cite{DMS2011}. Moreover, integrating by part again,
\[(T\Pi)^* f(x,u) \ = \ \int_\T \cos(v) \po V'(x)f(x,v) - \partial_x f(x,v)\pf\dd v\,,\]
so that
\begin{eqnarray*}
(T\Pi)^*T\Pi f(x,u) & = & \int_\T \cos(v)^2 \po V'(x)\partial_x - \partial_x^2\pf \int f(x,w)\dd w \dd v\\
& =& \frac12 \po V'(x)\partial_x - \partial_x^2\pf \Pi f(x,u)\,.
\end{eqnarray*} 
As a consequence,
\[\int_E \po T\Pi f\pf^2 \dd \mu \ = \ \frac12 \int_\R \po \partial_x \Pi f\pf^2 e^{-V} \ \geqslant \ \frac12C_P \int \po \Pi f -\mu f\pf^2 \dd \mu\,,\]
where $C_P$ is the Poincar\'e constant of the measure $e^{-V}$. This is the macroscopic coercivity condition $(H_2)$ of \cite{DMS2011}. In view of the expression of $(T\Pi)^*T\Pi $, the last condition $(H_4)$ of \cite{DMS2011} follows from \cite[Lemma 4]{DMS2011}. As a consequence, Assumption \ref{HypoDMS} holds.

Finally, let us prove that Assumption \ref{HypoHypoco1} holds. Denoting $\tilde U_t = U_t+\pi$, remark that
\[\left\{\begin{array}{rcl}
\dd X_t & = & -\cos(\tilde U_t)\dd t\\
\dd \tilde U_t & = &  V'(X_t) \sin(\tilde U_t)\dd t + \dd B_t\,.
\end{array}\right.\]
In other words $(X_t,U_t+\pi)_{t\geqslant 0}$ is a Markov process with generator $L^*$ (this is reminiscent of the change of variable $w=-v$ for the Langevin process). Denoting $p_t$ and $p_t^*$ the transition kernels of $L$ and $L^*$, we thus have, for all $ (x,u),(y,w)\in E$,
\[p_t\po (y,w),(x,u)\pf  \ = \ p_t^*\po (x,u),(y,w)\pf \  = \ p_t\po (x,u-\pi),(y,w-\pi)\pf \,.\]
Since $\mu$ is invariant by the transformation $(x,u)\mapsto (x,u-\pi)$, the densities $r_t=p_t/\mu$ and $r_t^*=p_t/\mu$ satisfy the same relation. From \cite[Theorem 1.5]{CattiauxhormanderI}, Assumption \ref{HypoHypoco2} implies that $p_t((x,u),(y,w)) \leqslant C(1\wedge t)^{-M}$ for some $C,M>0$ uniformly in $(x,u),(y,w)\in E$. Then we bound
\begin{eqnarray*}
\|r_t\po (x,u),\cdot\pf\|_2^2 & = & \int_E r_t\po (x,u),(y,w)\pf r_t\po (y,w+\pi),(x,u+\pi)\pf e^{-V(y)} \dd y \dd w  \\
& \leqslant & \frac{Ce^{V(x)}}{\po 1\wedge t\pf^M}  \int_E r_t\po (x,u),(y,w)\pf  e^{-V(y)} \dd y \dd w \ = \ \frac{Ce^{V(x)}}{\po 1\wedge t\pf^M} \,,
\end{eqnarray*}
which establishes Assumption \ref{HypoHypoco1}.

As a conclusion, Corollary \ref{CorLyap} holds for the process $(X_t,U_t)_{t\geqslant 0}$. Moreover,  Theorem \ref{TheoremHitting},  Inequality \eqref{EqTemps} and the quantitative bound on $\|r_t\po (x,u),\cdot\pf\|_2$ yield 
\[  \mathbb E_{(x,u)} \po e^{\theta T_U} \pf \leqslant \ C e^{\frac12 V(x)}\]
for some $C>0$ for all $x\in E$ all set $U$ with $\mu(U)>0$ and all $\theta<h\po \mu(U)\pf$.

\bigskip

\noindent\textbf{Acknowledgments.} This work has been partially supported by the Project EFI ANR-17-CE40-0030 of the French National Research Agency.

\bibliographystyle{plain}
\bibliography{biblio}